\documentclass[a4paper, oneside]{article}
\usepackage{enumitem} 
\usepackage{lmodern} 
\usepackage{gastex} 
\usepackage{amssymb,amsthm,amsmath} 
\usepackage{bm}
\usepackage[ansinew]{inputenc} 
\usepackage[dvips]{graphicx} 
\usepackage{verbatim}
\usepackage{tikz}
\usepackage[normalem]{ulem} 
\usepackage{authblk}

\usetikzlibrary{arrows}
\usetikzlibrary{decorations.pathreplacing}
\usetikzlibrary{automata}
\usetikzlibrary{positioning}
\theoremstyle{definition}
\newtheorem{definition}{Definition}[section] 
\newtheorem{lemma}[definition]{Lemma} 
\newtheorem{theorem}[definition]{Theorem}

\newtheorem{corollary}[definition]{Corollary}

\newtheorem{problem}[definition]{Problem}
\newtheorem{proposition}[definition]{Proposition}

\newcommand{\N}{\mathbb{N}}
\newcommand{\Npos}{{\mathbb{N}_+}}

\newcommand{\Z}{\mathbb{Z}}

\newcommand{\graph}[1]{{\mathcal{#1}}} 
\newcommand{\init}[1]{\iota(#1)} 
\newcommand{\term}[1]{\tau(#1)} 
\newcommand{\abs}[1]{{\left\vert #1 \right\vert}} 
\newcommand{\arr}{{\hbox{r}}}
\newcommand{\digs}{\Sigma} 

\newcommand{\bra}[1]{{\Sigma_{(#1)}}} 
\newcommand{\stud}[1]{{#1^{(*)}}} 

\DeclareMathOperator{\id}{Id} 
\DeclareMathOperator{\foll}{\omega} 
\DeclareMathOperator{\lbl}{\lambda} 
\DeclareMathOperator{\cyl}{Cyl} 

\begin{document}

\title{Direct Prime Subshifts and Canonical Covers}
\author{Johan Kopra}
\affil{Department of Mathematics and Statistics, \\FI-20014 University of Turku, Finland}
\affil{jtjkop@utu.fi}
\date{}
\maketitle

\setcounter{page}{1}

\begin{abstract}
\noindent We present a new sufficient criterion to prove that a non-sofic half-synchronized subshift is direct prime. The criterion is based on conjugacy invariant properties of Fischer graphs of half-synchronized shifts. We use this criterion to show as a new result that all $n$-Dyck shifts are direct prime, and we also give new proofs of direct primeness of non-sofic beta-shifts and non-sofic $S$-gap shifts. We also construct a class of non-sofic synchronized direct prime subshifts which additionally admit reversible cellular automata with all directions sensitive.
\end{abstract}

\providecommand{\keywords}[1]{\textbf{Keywords:} #1}
\noindent\keywords{half-synchronized subshifts, Fischer graphs, direct factorizations, sensitivity, cellular automata}

\section{Introduction}

Whenever a subshift $X$ can be represented as a product $Y\times Z$ (in the sense that $X$ is conjugate to $Y\times Z$), the dynamics of $X$ can be understood in terms of the dynamics of the simpler systems $Y$ and $Z$, and such systems $Y$ and $Z$ are called direct factors of $X$. If in all decompositions of $X$ into $Y\times Z$ either $Y$ or $Z$ is a trivial subshift, we say that $X$ is direct prime. Direct prime subshifts can thus be seen as building blocks of more general subshifts in a similar sense as prime numbers can be seen as building blocks of natural numbers.

In general determining whether a given subshift is direct prime or not seems to be a difficult problem. Lind gives sufficient conditions in~\cite{Lind84} for SFTs based on their entropies: for example any mixing SFT with entropy $\log p$ for a prime number $p$ is topologically direct prime. The paper of Meyerovitch~\cite{Mey17} contains results on multidimensional full shifts, multidimensional 3-colored chessboard shifts and $n$-Dyck shifts, a class of non-sofic half-synchronized shifts.

To approach the problem of determining whether a given subshift is direct prime we consider Fischer graphs, which are certain labeled directed graphs that are canonically associated to all half-synchronized subshifts. Broadly speaking, we would like to pinpoint some suitable property $P$ that necessarily holds in the Fischer graph of any half-synchronized subshift which is a product of two non-trivial subshifts. Then to prove that a half-synchronized subshift $X$ is not equal to $Y\times Z$ for nontrivial $Y$ and $Z$, we check that the Fischer graph of $X$ does not have the property $P$. One more obstacle remains: conjugate subshifts can in fact have different Fischer graphs, so to conclude from this that $X$ is not even conjugate to any $Y\times Z$, we would need to check that the property $P$ does not hold for the Fischer graph of any subshift that is conjugate to $X$. For this we need to choose the property $P$ so that it remains invariant between Fischer graphs of conjugate subshifts.

The concrete sufficient criterion that we present for showing that a half-synchronized non-sofic subshift (with a fixed point) is direct prime is based on choosing the property $P$ above as ``the Fischer graph of the subshift has a strictly proximal and eventually geodesic pair of infinite paths'' in Corollary~\ref{critCor}. We use this criterion to prove that $n$-Dyck shifts are direct prime for all integers $n>1$: previously this was known in the case when $n$ is a prime number~\cite{Mey17}. Using the same criterion we also give new proofs for the direct primeness of non-sofic $S$-gap shifts in Theorem~\ref{primeS} (which also follows from~\cite{Kop20glider} by using the argument of~\cite{Kop20}) and non-sofic beta-shifts in Theorem~\ref{primeBeta} (originally in~\cite{Kop20}).

Our original motivation for considering direct prime subshifts comes from the question of how the structure of a given subshift $X$ affects the range of possible dynamics of reversible cellular automata (RCA) on $X$. In Section~2.4 of~\cite{KopDiss} we argue that a relatively mild reasonable criterion for a CA to be dynamically complex is, in the terminology of directional dynamics of Sablik~\cite{Sab08}, that all its directions are sensitive. Depending on the subshift $X$ such RCA may exist (e.g. whenever $X$ is an infinite transitive sofic shift~\cite{Kop20glider}) or not (e.g. whenever $X$ a non-sofic beta-shift~\cite{Kop20} or a non-sofic $S$-gap shift~\cite{Kop20glider}). Since the existence of RCA with all directions sensitive on a subshift $X$ has been confirmed when $X$ is an infinite transitive sofic shift~\cite{Kop20glider}, the natural next step would to focus on the case when $X$ is a non-sofic synchronized subshift.

If a subshift is conjugate to a product $Y\times Z$ of two infinite transitive subshifts, then RCA with all directions sensitive exist for a trivial reason: the partial shift map $\tau: Y\times Z\to Y\times Z$ defined by $\tau(y,z)=(\sigma(y),z)$ ($\sigma$ is the usual shift map on $Y$) is such an RCA. Up to this point it has been unclear whether RCA with all directions sensitive can exist on any direct prime non-sofic synchronized subshift $X$ (in particular, in this case the partial shift map construction is unavailable). By using our new criterion for proving direct primeness we can present examples of such subshifts in Section~\ref{starStudSect}.

\section{Preliminaries}

In this section we recall some preliminaries concerning symbolic dynamics and topological dynamics in general. Good references to these topics are~\cite{Kur03,LM95}.

A (possibly infinite) nonempty set $A$ is called an \emph{alphabet}. For $n\in\Npos$ we have a special alphabet $\digs_n=\{0,1,\dots,n-1\}$. The set $A^\Z$ of bi-infinite sequences (\emph{configurations}) over $A$ is called a \emph{full shift}. Formally any $x\in A^\Z$ is a function $\Z\to A$ and the value of $x$ at $i\in\Z$ is denoted by $x[i]$. It contains finite, right-infinite and left-infinite subsequences denoted by $x[i,j]=x[i]x[i+1]\cdots x[j]$, $x[i,\infty]=x[i]x[i+1]\cdots$ and $x[-\infty,i]=\cdots x[i-1]x[i]$.

A configuration $x\in A^\Z$ (respectively, $x\in A^\N$) is \emph{periodic} if there is a $p\in\Npos$ such that $x[i+p]=x[i]$ for all $i\in\Z$ (respectively, $i\in\N$). Then we may also say that $x$ is $p$-periodic or that $x$ has period $p$. If $x$ is $1$-periodic, we call it a \emph{fixed point}. We say that $x\in A^\Z$ (respectively, $x\in A^\N$) is \emph{eventually periodic} if there are $p\in\Npos$ and $i_0\in\Z$ (respectively, $i_0\in\N$) such that $x[i+p]=x[i]$ holds for all $i\geq i_0$.

A \emph{subword} of $x\in A^\Z$ is any finite sequence $x[i,j]$ where $i,j\in\Z$, and we interpret the sequence to be empty if $j<i$. Any finite sequence $w=w[1] w[2]\cdots w[n]$ (also the empty sequence, which is denoted by $\epsilon$), where $w[i]\in A$, is a \emph{word} over $A$. Unless we consider a word $w$ as a subword of some configuration, we start indexing the symbols of $w$ from $1$ as we have done here. Similarly, right-infinite sequences are indexed $x=x[0]x[1]x[2]\cdots$ and left-infinite sequences are indexed $x=\cdots x[-3]x[-2]x[-1]$. If $w\neq\epsilon$, we say that $w$ occurs in $x$ at position $i$ if $x[i]\cdots x[i+n-1] = w[1]\cdots w[n]$. The concatenation of a word or a left-infinite sequence $u$ with a word or a right-infinite sequence $v$ is denoted by $uv$. A word $u$ is a \emph{prefix} of a word or a right-infinite sequence $x$ if there is a word or a right-infinite sequence $v$ such that $x=uv$. Similarly, $u$ is a \emph{suffix} of a word or a left-infinite sequence $x$ if there is a word or a left-infinite sequence $v$ such that $x=vu$. The set of all words over $A$ is denoted by $A^*$, and the set of non-empty words is $A^+=A^*\setminus\{\epsilon\}$. The set of words of length $n$ is denoted by $A^n$. For a word $w\in A^*$, $\abs{w}$ denotes its length, i.e. $\abs{w}=n\iff w\in A^n$. For any word $w\in A^+$ we denote by $w^\infty$ the right-infinite sequence obtained by infinite repetitions of the word $w$. We denote by $w^\Z$ the configuration defined by $w^\Z[in,(i+1)n-1]=w$ (where $n=\abs{w}$) for every $i\in\Z$.

Any collection of words $L\subseteq A^*$ is called a \emph{language}. For any set $S$ of configurations, right- or left-infinite sequences and finite words the collection of words appearing as subwords of elements of $S$ is the language of $S$, denoted by $L(S)$. For $n\in\N$ we denote $L^n(S)=L(S)\cap A^n$. For any $L,K\subseteq A^*$, let
\[LK=\{u,v\mid u\in L, v\in K\}\qquad L^*=\{w_1\cdots w_n\mid n\geq 0,w_i\in L\}\subseteq A^*.\]
If $\epsilon\notin L$, define $L^+=L^*\setminus\{\epsilon\}$ and if $\epsilon\in L$, define $L^+=L^*$.

To consider topological dynamics on subsets of the full shifts, the set $A^\Z$ is endowed with the product topology (with respect to the discrete topology on $A$). This is a metrizable space, which is also compact when $A$ is finite. The shift map $\sigma:A^\Z\to A^\Z$ is defined by $\sigma(x)[i]=x[i+1]$ for $x\in A^\Z$, $i\in\Z$, and it is a homeomorphism. If $X\subseteq A^\Z$ is such that $\sigma(X)=X$, we say that $X$ is \emph{shift-invariant}. If $A$ is finite, any topologically closed shift-invariant nonempty subset $X\subseteq A^\Z$ is called a \emph{subshift} or just a \emph{shift}. Alternatively, any subshift $X\subseteq A^\Z$ can be characterized by a list of \emph{forbidden words} $\mathcal{F}\subseteq A^+$ such that
\[X=\{x\in A^\Z\mid\mbox{No element of }w\in\mathcal{F}\mbox{ occurs in }x\}.\]
Any $w\in L(X)\setminus{\epsilon}$ and $i\in\Z$ determine a \emph{cylinder} of $X$
\[\cyl_X(w,i)=\{x\in X\mid w \mbox{ occurs in }x\mbox{ at position }i\}.\]

If $A$ and $B$ are alphabets, the elements of $A^n\times B^n$ can be naturally identified with the elements of $(A\times B)^n$ for all $n>0$, and elements of $A^\Z\times B^\Z$ can be identified with the elements of $(A\times B)^\Z$. Using these identifications we may say that $X=Y\times Z$ is a subshift whenever $Y$ and $Z$ are subshifts.

\begin{definition}
Let $X\subseteq A^\Z$ and $Y\subseteq B^\Z$ be arbitrary sets of configurations. We say that a map $F:X\to Y$ is a \emph{sliding block code} from $X$ to $Y$ (with memory $m$ and anticipation $a$ for integers $m\leq a$) if there exists a \emph{local rule} $f:A^{a-m+1}\to B$ such that $F(x)[i]=f(x[i+m],\dots,x[i],\dots,x[i+a])$. If $X=Y$ and $X$ is a subshift, we say that $F$ is a \emph{cellular automaton} (CA). If we can choose $m$ and $a$ so that $-m=a=r\geq 0$, we say that $F$ is a radius-$r$ CA.
\end{definition}

Note that both memory and anticipation can be either positive or negative. Note also that if $F$ has memory $m$ and anticipation $a$ with the associated local rule $f:A^{a-m+1}\to A$, then $F$ is also a radius-$r$ CA for $r=\max\{\abs{m},\abs{a}\}$, with possibly a different local rule $f':A^{2r+1}\to A$.

If $X$ and $Y$ are subshifts and if there is a bijective sliding block code $F:X\to Y$, we say that $F$ is a conjugacy and that $X$ is conjugate with $Y$ (via $F$). It is known that then the inverse map of $F$ is also a sliding block code. In particular, the inverse map of a bijective CA is also a CA, which is why they are also known as \emph{reversible} CA (RCA).

The notions of almost equicontinuity and sensitivity can be defined for general topological dynamical systems. We omit the topological definitions, because for cellular automata on transitive subshifts there are combinatorial characterizations for these notions using blocking words. We present these alternative characterizations below.

\begin{definition}
Let $F:X\to X$ be a radius-$r$ CA and $w\in L(X)$. We say that $w$ is a \emph{blocking word} if there is an integer $e$ with $\abs{w}\geq e\geq r+1$ and an integer $p\in[0,\abs{w}-e]$ such that
\[\forall x,y\in\cyl_X(w,0), \forall n\in\N, F^n(x)[p,p+e-1]=F^n(y)[p,p+e-1].\]
\end{definition}

The following is proved in Proposition~2.1 of~\cite{Sab08}.

\begin{proposition}\label{equiblock}
If $X$ is a transitive subshift and $F:X\to X$ is a CA, then $F$ is almost equicontinuous if and only if it has a blocking word.
\end{proposition}

We say that a CA on a transitive subshift is \emph{sensitive} if it is not almost equicontinuous. The notion of sensitivity is refined by Sablik's framework of directional dynamics~\cite{Sab08}.

\begin{definition}
Let $F:X\to X$ be a cellular automaton and let $p,q\in\Z$ be coprime integers, $q>0$. Then $p/q$ is a \emph{sensitive direction} of $F$ if $\sigma^p\circ F^q$ is sensitive. Similarly, $p/q$ is an \emph{almost equicontinuous direction} of $F$ if $\sigma^p\circ F^q$ is almost equicontinuous.
\end{definition}

For a subshift $X$ and a configuration $x\in X$ denote $x^+=x[0,\infty]$ and $x^-=x[-\infty,-1]$, so $x=x^-x^+$. Denote $X^+=\{x^+\mid x\in X\}$ and $X^-=\{x^-\mid x\in X\}$. For any $x\in X$ we define the follower set of $x^-$ by $\foll_X(x^-)=\{y^+\in X^+\mid x^- y^+\in X\}$ and for any $w\in L(X)$ we define the follower set of $w$ by $\foll_X(w)=\{x^+\in X^+\mid wx^+\in X^+\}$. The subscript $X$ may be dropped when it is clear from the context. By making use of follower sets we can define some natural classes of transitive subshifts.

\begin{definition}
We say that a subshift $X$ is \emph{transitive} (or \emph{irreducible} in the terminology of~\cite{LM95}) if for all words $u,v\in L(X)$ there is $w\in L(X)$ such that $uwv\in L(X)$. We say that $X$ is \emph{mixing} if for all $u,v\in L(X)$ there is $N\in\N$ such that for all $n\geq N$ there is $w\in L^n(X)$ such that $uwv\in L(X)$.
\end{definition}

\begin{definition}\label{halfsynchDef}
For a subshift $X$, we say that a word $w\in L(X)$ is \emph{half-synchronizing} if there is a sequence $x^{-}\in X^{-}$ with $x^{-}[-\abs{w},-1]=w$ satisfying $L(x^-)=L(X)$ and $\foll(x^-)=\foll(w)$. If $X$ is a transitive subshift that has a half-synchronizing word, we say that $X$ is \emph{half-synchronized}.
\end{definition}

\begin{definition}
For a subshift $X$, we say that a word $w\in L(X)$ is \emph{synchronizing} if all sequences $x^{-}\in X^{-}$ with $x^{-}[-\abs{w},-1]=w$ satisfy $\foll(x^-)=\foll(w)$. If $X$ is a transitive subshift that has a synchronizing word, we say that $X$ is \emph{synchronized}.
\end{definition}
In particular, synchronized subshifts are half-synchronized.

\begin{definition}
A subshift $X$ is \emph{sofic} if $\{\foll(x^-)\mid x\in X\}$ is a finite set.
\end{definition}

As mentioned in the introduction, we say that a subshift $Y$ is a direct factor of a subshift $X$ if there is a subshift $Z$ such that $X$ is conjugate to $Y\times Z$. We also say that a subshift $X$ is direct prime if $X$ being conjugate to $Y\times Z$ implies that either $Y$ or $Z$ is a trivial subshift (i.e. contains only one configuration). We make the simple observation that the class of half-synchronized subshifts is closed under taking direct factors.

\begin{lemma}\label{halfsynchprod}
If $Y\times Z$ is half-synchronized, then so are also $Y$ and $Z$.
\end{lemma}
\begin{proof}
The transitivity of $Y\times Z$ implies the transitivity of $Y$ and $Z$. If $Y\times Z$ is half-synchronized with some half-synchronizing word $(w_1,w_2)$ ($w_1\in L(Y)$ and $w_2\in L(Z)$ of equal length), then $w_1$ and $w_2$ are easily seen to be half-synchronizing words of $Y$ and $Z$ respectively.
\end{proof}

\section{Canonical covers}

To any subshift $X$ it is possible to associate covers, i.e. labeled directed graphs such that the labels of all bi-infinite paths on the graph form a dense set on $X$. Some of these covers turn out to be canonical in the sense that any RCA on $X$ can be ``lifted'' in a unique way to a sliding block code on the set of labels of bi-infinite paths of the cover. Two such covers are Krieger graphs and Fischer graphs from~\cite{FF92}. We will recall the definitions and basic results.

A (directed) graph is a pair $\graph{G}=(V,E)$, where $V$ is the set of vertices and $E$ is the set of edges or arrows. Both of these sets may be infinite. Each edge $e\in E$ starts at an initial vertex denoted by $\init{e}$ and ends at a terminal vertex denoted by $\term{e}$. A word $p\in E^+$ is a path on $\graph{G}$ if for every $1\leq i<\abs{p}$ it holds that $\term{p[i]}=\init{p[i+1]}$. Similarly one defines right-infinite, left-infinite and bi-infinite paths, and the collection of all bi-infinite paths on $\graph{G}$ is denoted by $P(\graph{G})$. The initial vertex $\init{p}$ of a finite or right-infinite path $p$ is equal to the initial vertex of the first edge of $p$, and the terminal vertex $\term{p}$ of a finite or a left-infinite path $p$ is equal to the terminal vertex of the last edge of $p$. The graph $\graph{G}$ is strongly connected if for every pair of vertices $v,w\in V$ there is a finite path $p$ with $\init{p}=v$ and $\term{p}=w$.

The tensor product of directed graphs $\graph{G}_1=(V_1,E_2)$ and $\graph{G}_2=(V_2,E_2)$ is $\graph{G}_1\times\graph{G}_2=(V_1\times V_2,E_1\times E_2)$, where $\init{e_1,e_2}=(\init{e_1},\init{e_2})$ and $\term{e_1,e_2}=(\term{e_1},\term{e_2})$ for $e_i\in E_i$. If each $e_i$ has a label $a_i$, then the edge $(e_1,e_2)$ has the label $(a_1,a_2)$.

The Krieger graph $\graph{K}_X=(V_X,E_X)$ of a subshift $X\subseteq A^\Z$ is the graph with the vertex set $V_X=\{\foll(x^-)\mid x\in X\}$ and for all $x^-\in X^-$ and $a\in A$ such that $x^- a\in X^-$ there is an edge from $\foll(x^-)$ to $\foll(x^- a)$ labeled by $a$. We denote the label of any edge $e$ by $\lbl_X(e)$. The map $\lbl_X$ naturally extends to a map $\lbl_X:P(\graph{K}_X)\to X$, which replaces each edge of a bi-infinite path by its label. It is easy to see that $\lbl_X(P(\graph{K}_X))=X$. It is also easy to see that for subshifts $X$, $Y$ and $Z$ satisfying $X=Y\times Z$ it holds that $\graph{K}_X=\graph{K}_Y\times\graph{K}_Z$.

If $X$ is half-synchronized with a half-synchronizing word $w\in L(X)$, then the Fischer graph $\graph{F}_X$ of $X$ is the maximal strongly connected subgraph of $\graph{K}_X$ containing the vertex $\foll(w)$ (which by Definition~\ref{halfsynchDef} is indeed a vertex of $\graph{K}_X$). The labeling map for this graph is the restriction of the map $\lbl_X$ of the previous paragraph. It is shown on pages 146--147 of~\cite{FF92} that for any pair $w,w'\in L(X)$ of half-synchronizing words the Krieger graph of $X$ contains a path from $\foll(w)$ to $\foll(w')$. From this it follows that $\graph{F}_X$ does not depend on the choice of the half-synchronizing word $w$, so we may speak of \emph{the} Fischer graph of $X$.

If $w\in L(X)$ is half-synchronizing, then a vertex $v$ of $\graph{K}_X$ belongs to $\graph{F}_X$ precisely if there is a path from $\foll(w)$ to $v$. Namely, if $p$ is a finite path from $\foll(w)$ on $\graph{K}_X$, then $w\lbl_X(p)$ is also a half-synchronizing word and $\term{p}=\foll(w\lbl_X(p))$, so by the previous paragraph there is also a path from $\term{p}$ to $\foll(w)$. By transitivity every word of $L(X)$ is a label of some path starting at $\foll(w)$, so the set $\lbl_X(\graph{F}_X)$ is dense in $X$.

If $\graph{F}_X$ is finite, then $\lbl_X(\graph{F}_X)$ is closed and therefore equal to $X$. Sofic subshifts are characterized as those subshifts that are equal to the set of labels of some finite directed graph, so a non-sofic half-synchronized subshift necessarily has an infinite Fischer graph.

It turns out that Fischer graphs respect products.

\begin{lemma}\label{prodFischer}
If $X$, $Y$ and $Z$ are half-synchronized subshifts such that $X=Y\times Z$, then $\graph{F}_X=\graph{F}_Y\times\graph{F}_Z$.
\end{lemma}
\begin{proof}
Let $w_1$ and $w_2$ be equal length half-synchronizing words of $Y$ and $Z$ respectively such that $(w_1,w_2)$ is a half-synchronizing word of $X$. Let also $y^-\in Y^-$ and $z^-\in Z^-$ be such that they have suffixes $w_1$ and $w_2$ respectively and that $L(x^-)=L(X)$ for $x^-=(y^-,z^-)\in X^-$. 

To show that $\graph{F}_X$ is a subgraph of $\graph{F}_Y\times\graph{F}_Z$, let $(v_1,v_2)$ be a vertex of $\graph{F}_X$, meaning that there is a path $p=(p_1,p_2)$ from $\foll_X(w_1,w_2)=(\foll_Y(w_1),\foll_Z(w_2))$ to $(v_1,v_2)$ in $\graph{K}_X$. Therefore $p_1$ and $p_2$ are paths from $\foll_Y(w_1)$ to $v_1$ and $\foll_Z(w_2)$ to $v_2$ in $\graph{K}_Y$ and $\graph{K}_Y$ respectively, so $(v_1,v_2)$ is a vertex of $\graph{F}_Y\times\graph{F}_Z$.

To show that $\graph{F}_Y\times\graph{F}_Z$ is a subgraph of $\graph{F}_X$, let $(v_1,v_2)$ be a vertex of $\graph{F}_Y\times\graph{F}_Z$, meaning that there are paths $p_1$ and $p_2$ from $\foll_Y(w_1)$ to $v_1$ and $\foll_Z(w_2)$ to $v_2$ in $\graph{K}_Y$ and $\graph{K}_Z$ respectively. Without loss of generality assume that $\abs{p_1}\leq\abs{p_2}$. Let $w_1'$ and $w_2'$ be suffixes of length $\abs{w_2 \lbl_Z(p_2)}$ of $y^-\lbl_Y(p_1)$ and $z^-\lbl_Z(p_2)$ respectively: they have suffixes $w_1\lbl_Y(p_1)$ and $w_2\lbl_Z(p_2)$ respectively. Then
\begin{flalign*}
& v_1=\foll_Y(y^-\lbl_Y(p_1))=\foll_Y(w_1\lbl_Y(p_1))=\foll_Y(w_1')\mbox{ and} \\
& v_2=\foll_Z(z^-\lbl_Z(p_2))=\foll_Z(w_2\lbl_Y(p_2))=\foll_Z(w_2')
\end{flalign*}
To show that $(v_1,v_2)$ is a vertex of $\graph{F}_X$, it is therefore sufficient to show that $(w_1',w_2')$ is a half-synchronizing word of $X$. This in turn follows after we show that $L(x'^-)=L(X)$ for $x'^-=(y^-\lbl_Y(p_1),z^-\lbl_Z(p_2))\in X^-$. Let therefore $u=(u_1,u_2)\in L(X)$ be arbitrary, with $u_1\in L(Y)$ and $u_2\in L(Z)$ of equal length $n$, and choose any $u_1'\in L(Y)$ and $u_2'\in L(Z)$ of length $k=\abs{p_2}-\abs{p_1}$ such that $u_1 u_1'\in L(Y)$ and $u_2'u_2\in L(Z)$. Because $X=Y\times Z$, it follows that $u'=(u_1 u_1',u_2'u_2)\in L(X)$ and therefore $u'=x^-[i,i+n+k-1]$ for some $i\in\Z$ such that $i+n+k-1<0$. Thus
\begin{flalign*}
& u_1=y^-[i,i+n-1]=(y^-\lbl_Y(p_1))[i-\abs{p_1},i+n-\abs{p_1}-1]\mbox{ and} \\
& u_2=z^-[i+k, i+n+k-1]=(z^-\lbl_Z(p_2))[i-\abs{p_1}, i+n-\abs{p_1}-1],
\end{flalign*}
so $u$ occurs in $x'^-$ at position $i-\abs{p_1}$.
\end{proof}

Let $C,D$ be disjoint alphabets (not necessarily finite), interpret $CD$ and $DC$ to be new alphabets and let $X\subseteq (CD)^\Z$, $Y\subseteq(DC)^\Z$ be shift invariant (not necessarily closed or compact). A bijective map $F:X\to Y$ is a forward bipartite code (for partition $(C,D)$) if the image of any $x\in X$ with $x[i]=c_i d_i$ ($c_i\in C$, $d_i\in D$) satisfies $F(x)[i]=d_{i}c_{i+1}$. Similarly, $F$ is a backward bipartite code (for partition $(C,D)$) if always $F(x)[i]=d_{i-1}c_i$. If $X$ and $Y$ are subshifts, then a bipartite code $F:X\to Y$ is a conjugacy. We recall the following.

\begin{theorem}[\cite{Nasu86}, Theorem~2.4]\label{decomp}
Every conjugacy between subshifts can be represented as a composition of bipartite codes and bijective symbol maps that are applied coordinatewise to configurations.
\end{theorem}

We say that a given map $F:X\to X$ lifts to a map $F':P(\graph{K}_X)\to P(\graph{K}_X)$ (resp. $F':P(\graph{F}_X)\to P(\graph{F}_X)$) if $\lbl_X(F'(x))=F(\lbl_X(x))$ for $x\in P(\graph{K}_X)$ (resp. for $x\in P(\graph{F}_X)$). It is known that bipartite codes between subshifts lift to bipartite codes between the path sets of Krieger graphs and Fischer graphs. We will recall the details of this.

\begin{definition}\label{indGraph}
Let $\graph{G}=(V,E)$ be a bipartite graph with partitions $V=V_1\cup V_2$ and $E=C\cup D$ such that each edge of $C$ goes from $V_1$ to $V_2$ and each edge of $D$ goes from $V_2$ to $V_1$. If $E_1\subseteq CD$ is the collection of paths of length 2 starting at $V_1$ and $E_2\subseteq DC$ is the collection of paths of length 2 starting at $V_2$, then $\graph{G}_1=(V_1,E_1)$ and $\graph{G}_2=(V_2,E_2)$ is an induced pair of graphs (of $\graph{G})$.
\end{definition}

The following lemma is essentially from~\cite{Nasu86}.

\begin{lemma}[\cite{Nasu86}, Corollary~4.6]\label{indKrieger}
Assume that there is a bipartite code $F:X\to Y$ between subshifts. Then $\graph{K}_X$ and $\graph{K}_Y$ are an induced pair of graphs up to renaming the edges and $F$ lifts to a bipartite code $F':P(\graph{K}_X)\to P(\graph{K}_Y)$.
\end{lemma}
\begin{proof}
Assume without loss of generality that $F$ is a forward bipartite code for partition $(A,B)$. Define the subshift
\begin{flalign*}
Z=&\{z\in (A\cup B)^\Z\mid\mbox{For some }x\in X\mbox{ and for all }i\in\Z, z[2i,2i+1]=x[i] \} \\
\cup&\{z\in (A\cup B)^\Z\mid\mbox{For some }y\in Y\mbox{ and for all }i\in\Z, z[2i,2i+1]=y[i] \}.
\end{flalign*}
In other words, we get $Z$ from configurations of $X$ and $Y$ by interpreting the symbols of $AB$ and $BA$ as pairs of symbols. It is then easily seen that $\graph{K}_X$ and $\graph{K}_Y$ are an induced pair of graphs of the Krieger graph $\graph{K}_Z$. Let $C$ be the set of edges of $\graph{K}_Z$ from $\graph{K}_X$ to $\graph{K}_Y$ and let $D$ be the set of edges from $\graph{K}_Y$ to $\graph{K}_X$: then the edges of $\graph{K}_X$ and $\graph{K}_Y$ can be renamed by elements of $CD$ and $DC$ respectively and the forward bipartite code $F':P(\graph{K}_X)\to P(\graph{K}_Y)$ between these graphs with renamed edges is clearly a lift of $F$.
\end{proof}

The map $F'$ of the previous lemma is in fact the unique continuous surjective lift of $F$, which is a consequence of the following theorem.

\begin{theorem}[\cite{FF92}, Theorems~2.11 and~2.12]\label{canonical}
A conjugacy $F:X\to Y$ between subshifts lifts to a unique continuous surjective map $F':P(\graph{K}_X)\to P(\graph{K}_Y)$ such that $\lbl_Y\circ F'=F\circ \lbl_X$. This map is invertible, uniformly continuous, and its inverse is uniformly continuous (i.e. both $F'$ and its inverse can be represented as sliding block codes). If $X$ and $Y$ are half-synchronized, the map $F'$ restricts to a bijection from $P(\graph{F}_X)$ to $P(\graph{F}_Y)$.
\end{theorem}

\begin{lemma}[\cite{Nasu86}, Corollary~4.8]\label{indFischer}
Assume that there is a bipartite code $F:X\to Y$ between half-synchronized subshifts. Then $\graph{F}_X$ and $\graph{F}_Y$ are an induced pair of graphs up to renaming the edges and $F$ lifts to a bipartite code $F':P(\graph{F}_X)\to P(\graph{F}_Y)$.
\end{lemma}
\begin{proof}
By Theorem~\ref{canonical} the map $F'$ of Lemma~\ref{indKrieger} restricts to a bijective bipartite code from $P(\graph{F}_X)$ to $P(\graph{F}_Y)$. The existence of this bipartite code guarantees that $\graph{F}_X$ and $\graph{F}_Y$ are an induced pair.
\end{proof}

\section{A Sufficient Criterion for Direct Primeness}

In this section we present a sufficient criterion of direct primeness for half-synchronized subshifts based on their Fischer graphs. First we define a few special types of paths on graphs.

\begin{definition}\label{qprox}
Let $\graph{G}=(V,E)$ be a graph and let $x,y\in P(\graph{G})$. We say that $x,y$ is a \emph{proximal pair} on $\graph{G}$ if for any $n\in\N$ there is $i\in\N$ such that  $x[i,i+n]=y[i,i+n]$ and a \emph{strictly proximal pair} if additionally $x[i]\neq y[i]$ for arbitrarily large $i\in\N$.
\end{definition}

\begin{definition}
A finite path $p\in E^+$ on a graph $\graph{G}=(V,E)$ is a \emph{geodesic} if it is a shortest path between vertices $\init{p}$ and $\term{p}$, and a right-infinite path $p\in E^\N$ is geodesic if all its finite subpaths are geodesics. A path $x\in P(\graph{G})$ is \emph{eventually geodesic} if there exists $i_0\in\Z$ such that $x[i,i+n]$ is a geodesic for all $i\geq i_0$ and $n\in\N$.
\end{definition}

It turns out that the property of being a strictly proximal pair or an eventually geodesic path is preserved under bipartite codes.

\begin{lemma}\label{geoproxLemma}
Let $\graph{G}_1=(V_1,E_1)$ and $\graph{G}_2=(V_2,E_2)$ be an induced pair of graphs of $\graph{G}=(V,E)$. Using the notation of Definition~\ref{indGraph}, let $X\in (CD)^\Z$ and $Y\in(DC)^\Z$ be the sets of bi-infinite paths on $\graph{G}_1$ and $\graph{G}_2$ and let $F:X\to Y$ be a bipartite code.
\begin{itemize}
\item If $x\in X$ is eventually geodesic, then $F(x)$ is eventually geodesic.
\item If $x\in X$, $y\in Y$ is a strictly proximal pair on $\graph{G}_1$, then $F(x),F(y)$ is a strictly proximal pair on $\graph{G}_2$.
\end{itemize}
\end{lemma}
\begin{proof}
We may assume that $F$ is a forward bipartite code, the case of backward bipartite code being similar. For every $i\in\Z$ let $x[i]=c_i d_i$ and $y[i]=c'_i d'_i$ for $c_i,c_i'\in C$ and $d_i,d'_i\in D$.

For the first item, assume to the contrary that $F(x)$ is not eventually geodesic. Then there is an arbitrarily large $i\in\N$ and a number $n\in\Npos$ such that $F(x)[i,i+n-1]=(d_i c_{i+1})\cdots(d_{i+n-1}c_{i+n})$ is not geodesic on $\graph{G}_2$. Let $w\in(DC)^m$ be a path of length $m<n$ on $\graph{G}_2$ with the same initial and terminal vertex as $F(x)[i,i+n-1]$. Then $c_i w d_{i+n}\in(CD)^{m+1}$ is a path of length $m+1$ on $\graph{G}_1$ with the same initial and terminal vertex as $x[i,i+n]$. The length of $x[i,i+n]$ is $n+1>m+1$, so it is not geodesic. Because $i$ could be chosen arbitrarily large, we see that $x$ is not eventually geodesic.

We proceed to the second item and assume that $x,y$ are a strictly proximal pair on $\graph{G}_1$. To check the proximality condition, let $n\in\N$ be arbitrary and let $i\in\N$ be such that $x[i,i+n]=(c_i d_i)\cdots(c_{i+n}d_{i+n})=y[i,i+n]$. Then $F(x)[i,i+n-1]=(d_i c_{i+1})\cdots(d_{i+n-1}c_{i+n})=F(y)[i,i+n-1]$ and proximality follows. To check strict proximality, let $i>0$ be an arbitrary coordinate such that $x[i]\neq y[i]$, so either $c_i\neq c'_i$ or $d_i\neq d_i'$. If $c_i\neq c'_i$, then $F(x)[i-1]=d_{i-1}c_i\neq d'_{i-1}c_i'=F(y)[i-1]$, and if $d_i\neq d_i'$ then $F(x)[i]=d_i c_{i+1}\neq d_i' c_{i+1}'=F(y)[i]$.
\end{proof}

From this lemma it then follows that the property of the Fischer graph having an eventually geodesic path or a strictly proximal pair is a conjugacy invariant.

\begin{theorem}
Assume that $X$ and $Y$ are conjugate half-synchronized subshifts. Then $\graph{F}_X$ has an eventually geodesic strictly proximal pair if and only if $\graph{F}_Y$ has such a pair.
\end{theorem}
\begin{proof}
The proof is by structural induction. By Theorem~\ref{decomp} it suffices to consider the cases when $X$ and $Y$ are conjugate either via a symbol map or a bipartite code. If $F:X\to Y$ is a bijective symbol map, then clearly $\graph{F}_X$ and $\graph{F}_Y$ are isomorphic graphs. If $F$ is a bipartite code, then by Lemma~\ref{indFischer} the graphs $\graph{F}_X$ and $\graph{F}_Y$ are an induced pair and there is a bipartite code $F':P(\graph{F}_X)\to P(\graph{F}_Y)$. If there is an eventually geodesic strictly proximal pair on $\graph{F}_X$, then by Lemma~\ref{geoproxLemma} there is an eventually geodesic strictly proximal pair also on $\graph{F}_Y$.
\end{proof}

First we prove a sufficient criterion for when a non-sofic half-synchronized subshift is ``almost'' direct prime in the sense that it cannot be represented as a product of two \emph{infinite} subshifts.

\begin{theorem}\label{crit}
If the Fischer graph of a non-sofic half-synchronized subshift $X$ does not contain an eventually geodesic strictly proximal pair and $X$ is conjugate to $Y\times Z$, then either $Y$ or $Z$ is finite.
\end{theorem}
\begin{proof}
Assume to the contrary that both $Y$ and $Z$ are infinite. By Lemma~\ref{halfsynchprod} the subshifts $Y$ and $Z$ are also half-synchronized. Because $X$ is not sofic, we may assume without loss of generality that $Y$ is not sofic and thus $\graph{F}_Y$ is infinite. Let $y$ be a path on $\graph{F}_Y$ such that $y[0,\infty]$ is geodesic. Because $Z$ is infinite and $\graph{F}_Z$ is strongly connected, there is some vertex $v$ with two different outgoing edges $e_1,e_2$. Let $w_1$ and $w_2$ be cycles on $\graph{F}_Z$ starting with $e_1$ and $e_2$ respectively. Let $u_1=w_1w_2$, $u_2=w_2w_1$: these are paths of equal length. Let $z_1,z_2$ be paths on $\graph{F}_Z$ such that $z_1[-\infty,-1],z_2[-\infty,-1]$ terminate at the vertex $v$ and $z_1[0,\infty]=u_1^\infty$, $z_2[0,\infty]=\prod_{i=1}^\infty (u_1^i u_2)$. The Fischer graph of $Y\times Z$ is equal to $\graph{F}_Y\times\graph{F}_Z$ by Lemma~\ref{prodFischer} and it contains paths $(y,z_1)$ and $(y,z_2)$. These are eventually geodesic and strictly proximal. By the previous theorem the graph $\graph{F}_X$ also has an eventually geodesic strictly proximal pair, a contradiction.
\end{proof}

After applying the previous theorem, to prove direct primeness it is sufficient to rule out the possibility of nontrivial finite direct factors. The existence of fixed points gives one way to do this, as seen in the next corollary.

\begin{corollary}\label{critCor}
If the Fischer graph of a non-sofic half-synchronized subshift $X$ does not contain an eventually geodesic strictly proximal pair and if $X$ contains a fixed point, then $X$ is topologically direct prime.
\end{corollary}
\begin{proof}
Assume to the contrary that $X$ is conjugate to a product $Y\times Z$ of non-trivial subshifts. By the previous theorem we may assume without loss of generality that $Y$ is finite with every configuration having period at most $p\geq 1$. The subshift $Y$ is also transitive and has a fixed point $a^\Z$ because $Y\times Z$ is transitive and has a fixed point. For any $u\in L(Y)$ there is $w\in L(Y)$ such that $uwa^p\in L(Y)$. But any configuration containing $uwa^p$ has period at most $p$, so $u\in a^*$. Therefore $L(Y)=a^*$ and $\abs{Y}=1$, contradicting its non-triviality.
\end{proof}

In the previous corollary it would be possible to replace the assumption of $X$ having a fixed point by the assumption of $X$ being mixing (by using the fact that a finite mixing subshift can contain only one configuration). We will not make use of this alternative criterion.

\section{Direct Primeness of Well Known Classes of Subshifts}

In this section we prove that several natural classes of subshifts are direct prime. First we consider the so-called $n$-Dyck shifts. It was shown in~\cite{Mey17} that the $n$-Dyck shift is topologically direct prime at least when $n$ is a prime number. We generalize this for all $n>1$. We first recall the basic definition from~\cite{Mey17}.

For a natural number $n>1$ we define the symbol sets $\bra{n,\ell}=\{\alpha_1,\dots,\alpha_n\}$, $\bra{n,\arr}=\{\beta_1,\dots,\beta_n\}$ (the left and right brackets) and $\bra{n}=\bra{n,\ell}\cup\bra{n,\arr}$. Let $M$ be the monoid generated by $\bra{n}\cup\{0\}$, with identity element $1$ and zero element $0$, subject to the relations $\alpha_i\beta_i=1$ and $\alpha_i\beta_j=0$ for $i,j\in\{1,\dots,n\}$, $i\neq j$.

The $n$-Dyck shift is defined by
\[D_n=\{x\in\bra{n}^\Z\mid x[i]\cdot x[i+1]\cdot \ldots \cdot x[i+k]\neq 0\mbox{ for all }i\in\Z, k\in\N\}:\]
intuitively this means that configurations of $D_n$ do not contain mismatched brackets. As an example we consider the simplest Dyck shift $D_2$, in which we replace the symbols $\alpha_1,\alpha_2,\beta_1,\beta_2$ by the brackets $(,[,)$ and $]$ respectively. For example $\cdots((()))\cdots\in D_2$, because any subword simplifies to an element of either $(^*$ or $)^*$. As another example, no element of $D_2$ can contain $((^n)^n]$ for any $n\in\N$, because $((^n)^n]=0$ as an element of $M$. This is seen by using the relation $()=1$ $n$ times and then the relation $(]=0$.

Let $w_i$, $i\in\N$ be an enumeration of all the words in $L(D_n)$ that are equal to $1$ as elements of the monoid $M$ (i.e. all the brackets are matched in $w_i$). Every word of $L(D_n)$ is a subword of some $w_i$. Any $w\in L(D_n)$ is a half-synchronizing word of $D_n$, since it is easily seen that $L(y^{-})=L(w)$ for $y^{-}=\cdots w_2 w_1 w_0 w$.

Note that the Fischer graph of $D_n$ consists of all the vertices and edges along paths starting from $\foll(\epsilon)$ in $\graph{K}_{D_n}$. It turns out that $\graph{F}_{D_n}=(V,E)$, where $V=\{\foll(w)\mid w\in\bra{n,\ell}^*\}$ and for $i\in\{1,\dots,n\}$
\begin{itemize}
\item from the vertex $\foll(\epsilon)$: there are self-loops $\beta_i$ and edges labeled by $\alpha_i$ to the vertex $\foll(\alpha_i)$
\item from any vertex of the form $\foll(w\alpha_i)$: there are edges labeled by $\alpha_j$ to the vertex $\foll(w\alpha_i\alpha_j)$ and an edge labeled by $\beta_i$ to the vertex $\foll(w)$.
\end{itemize}
We show a part of the Fischer graph $\graph{F}_{D_2}$ in Figure~\ref{Dyck}.

\begin{figure}
\centering
\begin{tikzpicture}[auto]
\node(e) at (0,0) [shape=circle,draw,minimum size=6mm] {$\epsilon$};

\node(a) at (4,1) [shape=circle,draw,minimum size=6mm] {$($};
\node(b) at (4,-1) [shape=circle,draw,minimum size=6mm] {$[$};

\node(aa) at (8,3) [shape=circle,draw,minimum size=6mm] {$(($};
\node(ab) at (8,1) [shape=circle,draw,minimum size=6mm] {$([$};
\node(ba) at (8,-1) [shape=circle,draw,minimum size=6mm] {$[($};
\node(bb) at (8,-3) [shape=circle,draw,minimum size=6mm] {$[[$};

\draw[-{triangle 45}] (e) to [in=170, out=90, loop above] node {$)$} ();
\draw[-{triangle 45}] (e) to [in=270, out=190, loop below] node {$]$} ();

\draw[-{triangle 45}] (e) to [out=30, in=180] node {$($} (a);
\draw[-{triangle 45}] (a) to node {$)$} (e);
\draw[-{triangle 45}] (e) to [out=330, in=180, below] node {$[$} (b);
\draw[-{triangle 45}] (b) to [above] node {$]$} (e);

\draw[-{triangle 45}] (a) to [out=70, in=180] node {$($} (aa);
\draw[-{triangle 45}] (aa) to [above] node {$)$} (a);
\draw[-{triangle 45}] (a) to [out=20, in=160]  node {$[$} (ab);
\draw[-{triangle 45}] (ab) to node {$]$} (a);
\draw[-{triangle 45}] (b) to [out=340, in=200, below]  node {$($} (ba);
\draw[-{triangle 45}] (ba) to [above] node {$)$} (b);
\draw[-{triangle 45}] (b) to [out=290, in=180, below] node {$[$} (bb);
\draw[-{triangle 45}] (bb) to [below] node {$]$} (b);

\node(x) at (9,3) {$\cdots$};
\node(y) at (9,1) {$\cdots$};
\node(x) at (9,-1) {$\cdots$};
\node(w) at (9,-3) {$\cdots$};
\end{tikzpicture}
\caption{The Fischer graph $\graph{F}_{D_2}$.}
\label{Dyck}
\end{figure}

\begin{theorem}\label{primeDyck}
The Dyck shift $D_n$ is topologically direct prime for every $n>1$.
\end{theorem}
\begin{proof}
We claim that $\graph{F}_{D_n}$ does not contain an eventually geodesic strictly proximal pair, so assume to the contrary that $x,y$ is such a pair. Since these paths are eventually geodesic, we may assume up to shifting the paths that $x[0,\infty]$ and $y[0,\infty]$ are geodesic. In particular there are infinitely many $\init{x[i]}$ for $i\in\N$, so it follows that $\lbl_{D_n}(x[i])\in\bra{n,\ell}$ for some $i\in\N$. Then $\lbl_{D_n}(x[j])\in\bra{n,\ell}$ for all $j>i$, because if $j>i$ were a minimal number such that $\lbl_{D_n}(x[j])\in\bra{n,\arr}$, it would follow that $\term{x[j]}=\init{x[j-1]}$, contradicting the geodesic assumption. We can argue similarly for the path $y$, so up to shifting the paths we may assume that $\lbl_{D_n}(x[i]),\lbl_{D_n}(y[i])\in\bra{n,\ell}$ for all $i\in\N$. Since $x,y$ is a strictly proximal pair, we additionally have $x[i]=y[i]$ (in particular $\term{x[i]}=\term{y[i]}$) and $x[i+1]\neq y[i+1]$ for some $i\in\N$, so after coordinate $i$ the paths $x$ and $y$ start following different branches in $D_n$ and $x[j]\neq y[j]$ for all $j>i$, a contradiction with proximality.

Since the Fischer graph of $D_n$ does not contain an eventually geodesic strictly proximal pair and $D_n$ contains a fixed point $\alpha_1^\Z$, it follows from Corollary~\ref{critCor} that $D_n$ is topologically direct prime.
\end{proof}

Next we will consider the so-called $S$-gap shifts. It is a previous result that every non-sofic $S$-gap shift is topologically direct prime. (In fact such $S$-gap shifts do not have RCA without almost equicontinuous directions~\cite{Kop20glider}: this is a stronger statement as seen in~\cite{Kop20}.) We present a new argument of this fact in the Fischer graph framework.

For nonempty $S\subseteq\N$, the $S$-gap shift $X_S\subseteq\digs_2^\Z$ is the subshift whose language $L(X_S)$ consists of all the subwords of elements of $\{01^n\mid n\in S\}^*$. Every $X_S$ is synchronized, because $0$ is a synchronizing word. By Theorem~3.4 of~\cite{DJ12} an $S$-gap shift is sofic if and only if $S$ is eventually periodic. In particular, for non-sofic $X_S$ the set $S$ is infinite, $L(X_S)$ contains $01^n$ for arbitrarily large $n\in\N$, and therefore $1^\Z\in X_S$

Note that the Fischer graph of $X_S$ consists of all the vertices and edges along paths starting from $\foll(0)$ in $\graph{K}_{X_S}$. It turns out that the Fischer graph $\graph{F}_{X_S}$ of a non-sofic $X_S$ has the vertex set $\{\foll(01^n)\mid n\in\N\}$, an edge labeled by $1$ from $\foll(01^n)$ to $\foll(01^{n+1})$, and for $n\in S$ an edge labeled by $0$ from $\foll(01^n)$ to $\foll(0)$.  In particular, from the fact that $X_S$ is not eventually periodic it follows that the follower sets of all $01^n$ are different.  We show a part of $\graph{F}_{X_S}$ in Figure~\ref{Sgap} in the case $S=\{2^i\mid i\in\N\}$.

\begin{figure}
\centering
\begin{tikzpicture}[auto]
\node(o) at (0,0) [shape=circle,draw,minimum size=8mm] {$0$};
\node(i) at (2,0) [shape=circle,draw,minimum size=8mm] {$01$};
\node(ii) at (4,0) [shape=circle,draw,minimum size=8mm] {$011$};
\node(iii) at (6,0) [shape=circle,draw,minimum size=8mm] {$01^3$};
\node(iv) at (8,0) [shape=circle,draw,minimum size=8mm] {$01^4$};
\node(v) at (10,0) [shape=circle,draw,minimum size=8mm] {$01^5$};

\draw[-{triangle 45}] (o) to [below] node {$1$} (i);
\draw[-{triangle 45}] (i) to [below] node {$1$} (ii);
\draw[-{triangle 45}] (ii) to [below] node {$1$} (iii);
\draw[-{triangle 45}] (iii) to [below] node {$1$} (iv);
\draw[-{triangle 45}] (iv) to [below] node {$1$} (v);

\draw[-{triangle 45}] (i) to [out=220, in=320, below] node {$0$} (o);
\draw[-{triangle 45}] (ii) to [out=150, in=30, above right] node {$0$} (o);
\draw[-{triangle 45}] (iv) to [out=140, in=40, above] node {$0$} (o);

\node(x) at (11,0) {$\cdots$};
\end{tikzpicture}
\caption{The Fischer graph $\graph{F}_{X_S}$.}
\label{Sgap}
\end{figure}
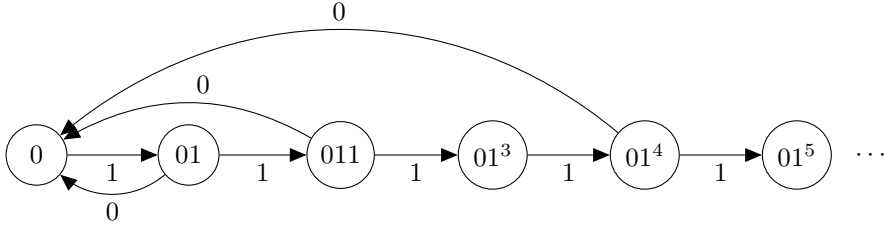

\begin{theorem}\label{primeS}
Every non-sofic $S$-gap shift $X_S$ is topologically direct prime.
\end{theorem}
\begin{proof}
We claim that $\graph{F}_{X_S}$ does not contain an eventually geodesic strictly proximal pair, so assume to the contrary that $x,y$ is such a pair. Since these paths are eventually geodesic, we may assume up to shifting the paths that $x[0,\infty]$ and $y[0,\infty]$ are geodesic. Then necessarily $\lbl_{X_S}(x[i])=\lbl_{X_S}(y[i])=1$ for $i\in\N$. On the other hand, since $x,y$ is a strictly proximal pair, we have $x[i]=y[i]$ and $x[i+1]\neq y[i+1]$ for some $i\in\N$. From $x[i]=y[i]$ it follows that $\init{x[i+1]}=\init{y[i+1]}$. This combined with $x[i+1]\neq y[i+1]$ implies that either $\lbl_{X_S}(x[i+1])=0$ or $\lbl_{X_S}(y[i+1])=0$, a contradiction.

Since the Fischer graph of $X_S$ does not contain an eventually geodesic strictly proximal pair and $X_S$ contains a fixed point $1^\Z$, it follows from Corollary~\ref{critCor} that $X_S$ is topologically direct prime.
\end{proof}

We conclude this section by considering the so-called beta-shifts. It is a previous result that every non-sofic beta-shift is topologically direct prime~\cite{Kop20}, but we can give a new proof of this fact in the Fischer graph framework. As in~\cite{Bla89} and Section~7.2.2. of~\cite{Lot02}, any real number $\beta>1$ can be associated in a certain way with a sequence $x_\beta\in \digs_n^\N$ (for some $n>1$) such that $x_\beta[0]>0$, $x_\beta\neq 10^\infty$ and every suffix of $x_\beta$ is lexicographically smaller than $x_\beta$ (with the lexicographical ordering $\leq$ induced from the usual ordering of $\digs_n$). Then the beta-shift (in base~$\beta$) is
\[X_\beta=\{x\in\digs_n^\Z\mid x[i,\infty]\leq x_\beta\mbox{ for every }i\in\Z\}.\]
This is non-sofic precisely when $x_\beta$ is not eventually periodic. We now consider only such cases.

It is stated in~\cite{FF92} that $X_\beta$ is half-synchronized and that every word in $L(X_\beta)$ is half-synchronized. The graph $\graph{F}_{X_\beta}$ then consists of all the vertices and edges along paths starting from $\foll(\epsilon)$ in $\graph{K}_{X_\beta}$. As mentioned in~\cite{DS19}, it turns out that the Fischer graph $\graph{F}_{X_\beta}$ of a non-sofic $X_\beta$ has the vertex set $\{\foll(x_\beta[0,n])\mid n\geq -1\}$, an edge labeled by $x_\beta[n+1]$ from $\foll(x_\beta[0,n])$ to $\foll(x_\beta[0,n+1])$, and for $0\leq i<x_\beta[n+1]$ an edge labeled by $i$ from $\foll(x_\beta[0,n])$ to $\foll(\epsilon)$.  In particular, from the fact that $x_\beta$ is not eventually periodic it follows that the follower sets of all $x_\beta[0,n]$ are different. We show a part of $\graph{F}_{X_\beta}$ in Figure~\ref{Bshift} in the case $x_\beta=22102\cdots$.

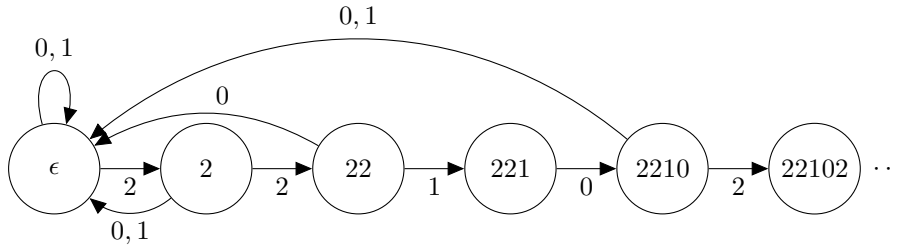
\begin{figure}
\centering
\begin{tikzpicture}[auto]
\node(o) at (0,0) [shape=circle,draw,minimum size=12mm] {$\epsilon$};
\node(i) at (2,0) [shape=circle,draw,minimum size=12mm] {$2$};
\node(ii) at (4,0) [shape=circle,draw,minimum size=12mm] {$22$};
\node(iii) at (6,0) [shape=circle,draw,minimum size=12mm] {$221$};
\node(iv) at (8,0) [shape=circle,draw,minimum size=12mm] {$2210$};
\node(v) at (10,0) [shape=circle,draw,minimum size=12mm] {$22102$};

\draw[-{triangle 45}] (o) to [below] node {$2$} (i);
\draw[-{triangle 45}] (i) to [below] node {$2$} (ii);
\draw[-{triangle 45}] (ii) to [below] node {$1$} (iii);
\draw[-{triangle 45}] (iii) to [below] node {$0$} (iv);
\draw[-{triangle 45}] (iv) to [below] node {$2$} (v);

\draw[-{triangle 45}] (i) to [out=220, in=320, below] node {$0,1$} (o);
\draw[-{triangle 45}] (ii) to [out=150, in=30, above right] node {$0$} (o);
\draw[-{triangle 45}] (iv) to [out=140, in=40, above] node {$0,1$} (o);

\node(x) at (11,0) {$\cdots$};

\draw[-{triangle 45}] (o) to [in=170, out=90, loop above] node {$0,1$} ();
\end{tikzpicture}
\caption{The Fischer graph $\graph{F}_{X_\beta}$.}
\label{Bshift}
\end{figure}

\begin{theorem}\label{primeBeta}
Every non-sofic beta-shift is topologically direct prime.
\end{theorem}
\begin{proof}
We claim that $\graph{F}_{X_\beta}$ does not contain an eventually geodesic strictly proximal pair, so assume to the contrary that $x,y$ is such a pair. Since these paths are eventually geodesic, we may assume up to shifting the paths that $x[0,\infty]$ and $y[0,\infty]$ are geodesic. Then necessarily $\term{x[i]}\neq\epsilon$ and $\term{y[i]}\neq\epsilon$ for $i\in\N$. On the other hand, since $x,y$ is a strictly proximal pair, we have $x[i]=y[i]$ and $x[i+1]\neq y[i+1]$ for some $i\in\N$. From $x[i]=y[i]$ it follows that $\init{x[i+1]}=\init{y[i+1]}$. This combined with $x[i+1]\neq y[i+1]$ implies that either $\term{x[i+1]}=\epsilon$ or $\term{y[i+1]}=\epsilon$, a contradiction.

Since the Fischer graph of $X_\beta$ does not contain an eventually geodesic strictly proximal pair and $X_\beta$ contains a fixed point $0^\Z$, it follows from Corollary~\ref{critCor} that $X_\beta$ is topologically direct prime.
\end{proof}

\section{A New Class of Direct Prime Subshifts}\label{starStudSect}

Up to this point we have applied Theorem~\ref{crit} to previously known classes of subshifts. The theorem can also be used in constructing new direct prime subshifts with some additional desirable properties. As an example we present a construction that shows that a direct prime non-sofic synchronized subshift can have RCA with only sensitive directions: recall from the introduction that the existence of such RCA is trivial if the subshift is a product of two infinite subshifts.

\begin{definition}\label{studDef}
For any subshift $X\subseteq A^\Z$ there is a star-studded version $\stud{X}$ containing a new symbol $*$, consisting of all configurations created by taking an arbitrary $x\in X$ and interleaving the symbol $*$ so that no two consecutive $*$ occur. To be more precise, let $h:(A\cup\{*\})^*\to A^*$ be a substitution determined by $h(*)=\epsilon$ and $h(a)=a$ for $a\in A$. If $X$ is characterized by a set $\mathcal{F}\subseteq A^*$ of forbidden words, then $\stud{X}$ is the subshift over $A\cup\{*\}$ characterized by the set of forbidden words
\[\mathcal{F}'=\{**\}\cup\{w\in (A\cup\{*\})^*\mid h(w)\in\mathcal{F}\}.\]
\end{definition}

Star-studded subshifts have RCA that resemble partial shift maps but do not rely on direct factorizations. These turn out to have all directions sensitive.

\begin{definition}\label{starCAdef}
For any subshift $X$ define the RCA $F_X:\stud{X}\to\stud{X}$ by
\begin{flalign*}
F_X(x)[i]=
\left\{
\begin{array}{l}
*\text{ when } x[i]=*, \\
x[i+1]\text{ when } x[i]x[i+1]\in A^2, \\
x[i+2]\mbox{ when } x[i]\in A \mbox{ and } x[i+1]=*.
\end{array}
\right.
\end{flalign*}
The CA $F_X$ fixes all occurrences of the symbol $*$ in configurations and shifts any occurrence of any symbol $a\in A$ to the left so that the symbol $a$ skips over any occurence of $*$.
\end{definition}

\begin{proposition}\label{starCA}
If $X\subseteq A^\Z$ is an infinite transitive subshift then $F_X:\stud{X}\to\stud{X}$ has all directions sensitive. 
\end{proposition}
\begin{proof}
To see that $F_X$ has all directions sensitive, assume to the contrary that there is an almost equicontinuous direction $p/q$ for coprime integers $p$ and $q$ such that $q>0$. This means that $G=\sigma^p\circ F_X^q$ is almost equicontinuous and admits a blocking word $w\in L(\stud{X})$.

Assume first that $p\geq 0$. If $x,y\in\cyl(w,0)$ are such that $x[-\infty,-1]=y[-\infty,-1]$, then $G^n(x)[-\infty,-1]=G^n(y)[-\infty,-1]$ for all $n\geq 0$ because $w$ is blocking. Choose such $x,y$ so that additionally $x[\abs{w},\infty],y[\abs{w},\infty]\in X^+$ and $x[i]\neq y[i]$ for some $i\geq\abs{w}$: this can be done because $X$ is infinite and transitive. From $p\geq 0$ it follows that $G$ shifts any occurrence of any symbol $a\in A$ to the left at a non-zero speed which is equal in both $x$ and $y$ and therefore $G^n(x)[-\infty,-1]\neq G^n(y)[-\infty,-1]$ for some $n>0$, a contradiction.

Assume then that $p<0$. If $x,y\in\cyl(w,0)$ are such that $x[\abs{w},\infty]=y[\abs{w},\infty]$, then $G^n(x)[\abs{w},\infty]=G^n(y)[\abs{w},\infty]$ for all $n\geq 0$ because $w$ is blocking. Choose such $x,y$ so that additionally $x[-2]=*\neq y[-2]$.  From $p<0$ it follows that $G$ shifts any occurrence of $*$ to the right at a constant non-zero speed and therefore $G^n(x)[j]=*\neq G^n(y)[j]$ for some $n>0$ and $j\geq\abs{w}$, a contradiction.
\end{proof}

Now that we have seen that star-studded subshifts have RCA with all directions sensitive, it remains to show that this class of subshifts contains direct prime non-sofic synchronized subshifts. First we show that to some extent the star-studded subshift inherits properties from the original one.

\begin{lemma}\label{distinctV}
If $X\subseteq A^\Z$ is a subshift and if $x_1^-,x_2^-\in X^-$, then $\foll_\stud{X}(x_1^-)\neq\foll_\stud{X}(x_2^-*)$. If additionally $\foll_X(x_1^-)\neq\foll_X(x_2^-)$, then $\foll_\stud{X}(x_1^-)\neq\foll_\stud{X}(x_2^-)$ and $\foll_\stud{X}(x_1^-*)\neq\foll_\stud{X}(x_2^-*)$.
\end{lemma}
\begin{proof}
Let $x\in X$ be such that $x^+\in\foll_X(x_1^-)$. Then $*x^+\in \foll_X(x_1^-)$ but $*x^+\notin\foll_\stud{X}(x_2^-*)$.

If $\foll_X(x_1^-)\neq\foll_X(x_2^-)$, then
\begin{flalign*}
& \foll_\stud{X}(x_1^-)\cap A^\N=\foll_X(x_1^-)\neq\foll_X(x_2^-)=\foll_\stud{X}(x_2^-)\cap A^\N \mbox{ and}\\
& \foll_\stud{X}(x_1^-*)\cap A^\N=\foll_X(x_1^-)\neq\foll_X(x_2^-)=\foll_\stud{X}(x_2^-*)\cap A^\N.
\end{flalign*}
\end{proof}

\begin{lemma}\label{studInherit}
If $X$ is half-synchronized, synchronized or non-sofic then $\stud{X}$ is half-synchronized, synchronized or non-sofic respectively. A half-syncronizing word of $X$ is also a half-synchronizing word of $\stud{X}$. 
\end{lemma}
\begin{proof}
It is easy to see that if $X$ is transitive, then $\stud{X}$ is also transitive.

If $X$ is half-synchronized, then choose any half-synchronizing $w\in L(X)$ and a sequence $x^{-}\in X^{-}$ with $x^{-}[-\abs{w},-1]=w$ satisfying $L(x^-)=L(X)$ and $\foll_X(x^-)=\foll_X(w)$. We claim that $w$ is a half-synchronizing word also in $\stud{X}$. To see this, note that if $y^-\in\stud{X}^-$ is formed by inserting stars $*$ into $x^-$ so that $y^-[-\abs{w}-1]=w$ (and $y^-[-1]\neq*$ if $w=\epsilon$), then automatically $\foll_\stud{X}(y^-)=\foll_\stud{X}(x^-)=\foll_\stud{X}(w)$. Therefore all we need to do is insert the stars so that $L(y^-)=L(\stud{X})$. This can be done, because any word $u\in L(X)$ actually occurs in $x^-$ infinitely many times, so it is possible to insert stars into the occurences of $u$ in $x^-$ in all the possible ways.

If $X$ is synchronized, then choose any synchronizing $w\in L(X)\setminus\{\epsilon\}$. We claim that $w$ is a synchronizing word also in $\stud{X}$. Namely, if $y^-\in\stud{X}^-$ with $y^-[-\abs{w}-1]=w$, then it can be formed by inserting stars $*$ into some $x^-\in X^-$ such that $x^-[-\abs{w}-1]=w$, and $\foll_X(x^-)=\foll_X(w)$ because $w$ is synchronizing. Therefore also $\foll_\stud{X}(y^-)=\foll_\stud{X}(x^-)=\foll_\stud{X}(w)$.

If $X\subseteq A^\Z$ is non-sofic, then $\{\foll_X(x^-)\mid y\in X\}$ is infinite. If $x_1,x_2\in X$ are such that $\foll_X(x_1^-)\neq\foll_X(x_2^-)$, then also $\foll_\stud{X}(x_1^-)\neq \foll_\stud{X}(x_2^-)$ by Lemma~\ref{distinctV}, so $\{\foll_\stud{X}(x^-)\mid x\in \stud{X}\}$ is infinite and $\stud{X}$ is non-sofic.
\end{proof}

We make the following definition to make sense of what Krieger graphs and Fischer graphs of star-studded shifts look like.

\begin{definition}
For any labeled graph $\graph{G}=(V,E)$ the star-studded version of $\graph{G}$ is $\stud{\graph{G}}=(V\cup \stud{V},E\cup \stud{E})$ with $V,\stud{V}$disjoint and $E,\stud{E}$ disjoint as follows. Let $\stud{V}=\{\stud{v}\mid v\in V\}$ be a copy of $V$, let $\stud{E}$ contain an edge labeled by $*$ from $v$ to $\stud{v}$ for each $v\in V$ and for any edge in $E$ labeled by $a$ from $v$ to $w$ for some $v,w\in V$ let $\stud{E}$ contain an edge labeled by $a$ from $\stud{v}$ to $w$.
\end{definition}

This definition has the nice property that the non-existence of eventually geodesic strictly proximal pairs in a graph is inherited by the star-studded version of the graph.

\begin{lemma}\label{studProx}
If a labeled graph $\graph{G}$ does not contain an eventually geodesic strictly proximal pair, then neither does $\stud{\graph{G}}$.
\end{lemma}
\begin{proof}
Let $\graph{G}=(V,E)$ and denote by $\lbl(e)$ the label of an edge $e$ in $\stud{\graph{G}}$. Assume to the contrary that $x,y$ is an eventually geodesic strictly proximal pair on $\stud{\graph{G}}$. We may assume up to shifting these paths that $x[0,\infty]$ and $y[0,\infty]$ are geodesic. If $\lambda_X(x[i])=*$ for some $i\in\N$, then necessarily $\init{x[i]}=v$, $\term{x[i]}=\stud{v}$ for some $v\in V$ and $\term{x[i+1]}=w$ for some $w\in V$ such that $E$ contains an edge $e$ from $v$ to $w$. But then $e$ is a shorter path from $v$ to $w$ than $x[i,i+1]$, contradicting the assumption of geodesicness. Thus the label of $x[0,\infty]$ does not contain occurrences of $*$ and similarly $y[0,\infty]$ does not contain occurrences of $*$. Therefore the paths $x[1,\infty],y[1,\infty]$ are contained completely in the subgraph $\graph{G}=(V,E)$ and $\graph{G}$ contains an eventually geodesic strictly proximal pair, a contradiction.
\end{proof}

To apply the previous lemma to Fischer graphs of star-studded subshifts, we need to show that star-studded versions of Fischer graphs are isomorphic to Fischer graphs of star-studded subshifts.

By Lemma~\ref{distinctV} the vertices $\foll_X(x^-)$ of $V_X$ are in bijective correspondence with the vertices in $V_{X,1}=\{\foll_\stud{X}(x^-)\mid x\in X\}\subseteq V_{\stud{X}}$ and the vertices $\stud{\foll_X(x^-)}$ of $\stud{V_X}$ (copies of $\foll_X(x^-)$) are in bijective correspondence with the vertices in $V_{X,2}=\{\foll_\stud{X}(x^-*)\mid x\in X\}\subseteq V_{\stud{X}}$. Also by Lemma~\ref{distinctV} the sets $V_{X,1}$ and $V_{X,2}$ are disjoint. Therefore  the vertices $V_{X,1}\cup V_{X,2}$ of $\graph{K}_\stud{X}$ are in a natural bijective correspondence with $V_X\cup\stud{V_X}$. The edges between the vertices of $V_{X,1}\cup V_{X,2}$ are clearly such that $\stud{\graph{K}_X}$ is a subgraph of $\graph{K}_\stud{X}$ (up to renaming the vertices). It is also clear that $\graph{K}_\stud{X}$ contains no other types of vertices, so in fact $\stud{\graph{K}_X}$ is isomorphic to $\graph{K}_\stud{X}$ (up to renaming the vertices). We denote this graph isomorphism by $\phi_X:\stud{\graph{K}_X}\to \graph{K}_\stud{X}$ and we may denote the labeling function of both of these graphs by $\lbl_X$.

\begin{lemma}\label{studFischer}
For every half-synchronized $X$ the graph $\stud{\graph{F}_X}$ is isomorphic to $\graph{F}_\stud{X}$.
\end{lemma}
\begin{proof}
Let $w\in L(X)$ be a half-synchronizing word of $X$, so it is also a half-synchronizing word of $\stud{X}$ by Lemma~\ref{studInherit}. Fix some $x^-\in X^-$ such that $\foll_X(w)=\foll_X(x^-)$. The set $\foll_X(w)$ is a vertex in $\stud{\graph{K}_X}$ (because it contains $\graph{K}_X$ as a subgraph) and $\phi_X(\foll_X(w))=\phi_X(\foll_X(x^-))=\foll_\stud{X}(x^-)=\foll_\stud{X}(w)$. The Fischer graph $\graph{F}_\stud{X}$ is the maximal strongly connected subgraph of $\graph{K}_{\stud{X}}$ containing $\foll_\stud{X}(w)$, so $\graph{F}_\stud{X}$ is isomorphic to the maximal strongly connected subgraph of $\stud{\graph{K}_X}$ containing $\phi_X^{-1}(\foll_\stud{X}(w))=\foll_X(w)$. This consists of the graph $\graph{F}_X$ together with the vertices of $\stud{V_X}$ and edges of $\stud{E_X}$ reachable from $\graph{F}_X$, and these form the graph $\stud{\graph{F}_X}$.
\end{proof}

Now we may apply Lemma~\ref{studProx} to Fischer graphs. After that we will prove the main results of this section.

\begin{lemma}\label{studFProx}
If the Fischer graph of a half-synchronized subshift $X$ does not contain an eventually geodesic strictly proximal pair, then neither does the Fischer graph of $\stud{X}$.
\end{lemma}
\begin{proof}
By Lemma~\ref{studProx} the graph $\stud{\graph{F}_X}$ does not contain an eventually geodesic strictly proximal pair. By Lemma~\ref{studFischer} the graph $\graph{F}_\stud{X}$ is isomorphic to $\stud{\graph{F}_X}$, so $\graph{F}_\stud{X}$ does not contain an eventually geodesic stricly proximal pair.
\end{proof}

\begin{theorem}
If a non-sofic half-synchronized subshift $X$ has a fixed point and $\graph{F}_X$ does not contain an eventually geodesic strictly proximal pair, then $\stud{X}$ is direct prime.
\end{theorem}
\begin{proof}
If $x$ is a fixed point of $X$, then it is also a fixed point of $\stud{X}$. By the previous lemma $\graph{F}_\stud{X}$ does not contain an eventually geodesic strictly proximal pair, so $\stud{X}$ is direct prime by Corollary~\ref{critCor}.
\end{proof}

Concrete examples can be obtained e.g. from $S$-gap shifts.

\begin{proposition}\label{SgapStar}If $X_S$ is a non-sofic $S$-gap shift, then $\stud{X_S}$ is a non-sofic synchronized subshift which is topologically direct prime and that admits a RCA with all directions sensitive.
\end{proposition}
\begin{proof}
Since $X_S$ is non-sofic and synchronized, it follows from Lemma~\ref{studInherit} that also $\stud{S_\beta}$ is non-sofic and synchronized. Since $X$ contains a constant configuration $1^\Z$ and in the proof of Theorem~\ref{primeS} it is shown that $\graph{F}_{X_S}$ does not contain an eventually geodesic strictly proximal pair, it follows from the previous theorem that $\stud{X_S}$ is topologically direct prime. By Proposition~\ref{starCA} the subshift $\stud{X_S}$ has a RCA with all directions sensitive.
\end{proof}

\section{Conclusions}

We have presented a new sufficient criterion that can be used to show that a non-sofic half-synchronized subshift is direct prime. We then applied the result to several natural classes of non-sofic half-synchronized subshifts. We expect that sharper criteria can be found.

\begin{problem}
Find a complete characterization (based on examining the Fischer graph) of direct prime half-synchronized subshifts.
\end{problem}

Unfortunately our criterion cannot make any distinctions between transitive sofic subshifts, because all the Fischer graphs of transitive sofic subshifts are finite and therefore they do not have eventually geodesic paths.

\begin{problem}
Find useful sufficient criteria (based on examining the Fischer graph) that can be used to show that a transitive sofic subshift is direct prime.
\end{problem}

As a first step it would be nice to find such a criterion that is able to distinguish the one vertex graph with two loops from the one vertex graph with six loops, because they are the Fischer graphs of the full shift $\Sigma_2^\Z$ (which is direct prime~\cite{Lind84}) and the full shift $\Sigma_6^\Z$ (which is not direct prime~\cite{Lind84}) respectively.

We have also given examples of non-sofic synchronized subshifts which are direct prime but that still have RCA with all directions sensitive. This is in contrast with previous examples of non-sofic half-synchronized subshifts without RCA with all directions sensitive~\cite{Kop20,Kop20glider}, which might have given the impression that such RCA can exist on a non-sofic synchronized subshift only when it can be represented as a product of two infinite subshifts. At this point it is not clear what kind of an answer one should expect to the following problem.

\begin{problem}
Characterize the non-sofic synchronized subshifts that admit RCA with only sensitive directions.
\end{problem}

One may also ask whether the property of all directions being sensitive is an appropriate minimal criterion for a CA to be dynamically complex. Indeed, it is still possible to ``extract'' trivial dynamics of an infinite set from the CA $F_X$ presented in Definition~\ref{starCAdef} in the following precise sense. The list of forbidden words $\mathcal{F}=\{11\}$ determines the golden mean subshift $X_\mathcal{F}\subseteq\digs_2^\Z$. Given any subshift $X\subseteq A^\Z$, the symbol map $a\mapsto 0$ ($a\in A$) and $*\mapsto 1$ extends to a surjective sliding block code $\phi:\stud{X}\to X_\mathcal{F}$. Then the identity map CA $\id:X_\mathcal{F}\to X_\mathcal{F}$ is a \emph{factor CA} of $F_X$ (via $\phi$), meaning that $\phi$ is surjective and $\phi{\;\circ\;} F_X=\id {\;\circ\;} \phi$. Similarly the partial shift map $\tau: Y\times Z\to Y\times Z$ defined by $\tau(y,z)=(\sigma(y),z)$ has the identity map on $Z$ as a factor CA.

\begin{problem}
Does there exist a non-sofic synchronized subshift $X$ (or even a more general transitive non-sofic subshift $X$) and a RCA $F:X\to X$ such that whenever $F':Y\to Y$ is a factor CA of $F$ on an infinite subshift $Y$, then $F'$ has all directions sensitive? 
\end{problem}

\section*{Acknowledgements}
The work was supported by the Finnish Cultural Foundation.

\bibliographystyle{plain}
\bibliography{mybib}{}

\end{document}